\documentclass[a4paper]{amsart}
\usepackage[english]{babel}
\usepackage{amsmath,amssymb,amsthm}

\usepackage[pdftex]{color}
\usepackage{graphicx}

\usepackage[bookmarks=true,hyperindex,pdftex,colorlinks,citecolor=blue]
{hyperref}

\hypersetup{pdftitle={On the Bishop-Phelps-Bollob\'{a}s property for numerical radius},
pdfauthor={Sun Kwang Kim, Han Ju Lee, and Miguel Martin},pdfpagemode=UseOutlines}

\theoremstyle{plain}
\newtheorem{theorem}{Theorem}[section]
\newtheorem{corollary}[theorem]{Corollary}
\newtheorem{prop}[theorem]{Proposition}
\newtheorem{proposition}[theorem]{Proposition}
\newtheorem{lemma}[theorem]{Lemma}

\theoremstyle{definition}

\newtheorem{example}[theorem]{Example}
\newtheorem{examples}[theorem]{Examples}
\newtheorem{definition}[theorem]{Definition}

\newtheorem{notation}[theorem]{Notation}

 \DeclareMathOperator{\re}{Re\,}
 \DeclareMathOperator{\e}{e}
 
 \DeclareMathOperator{\dist}{dist\,}

\newcommand{\C}{\mathbb{C}}
\newcommand{\R}{\mathbb{R}}
\newcommand{\N}{\mathbb{N}}

\newcommand{\inner}[1]{\ensuremath{\left\langle #1\right\rangle}}

\newcommand{\norm}[1]{\ensuremath{\lVert#1\rVert}}

\newcommand{\nor}[1]{\ensuremath{\left\|#1\right\|}}

\newcommand{\eps}{\varepsilon}

\newcommand{\nuu}{\textrm{nu}}

\renewcommand{\leq}{\leqslant}
\renewcommand{\geq}{\geqslant}

\begin{document}

\begin{center}
\small{ [ \emph{Abstract and Applied Analysis}, Volume 2014, Article ID 479208, 16 pages ] \linebreak http://dx.doi.org/10.1155/2014/479208 }
\end{center}

\title{On the Bishop-Phelps-Bollob\'{a}s property for numerical radius}

\author[Kim]{Sun Kwang Kim}
\address[Kim]{Department of Mathematics, Kyonggi University, Suwon 443-760, Republic of Korea}
\email{\texttt{lineksk@gmail.com}}

\author[Lee]{Han Ju Lee}
\address[Lee]{Department of Mathematics Education,
Dongguk University - Seoul, 100-715 Seoul, Republic of Korea}
\email{\texttt{hanjulee@dongguk.edu}}

\author[Mart\'{\i}n]{Miguel Mart\'{\i}n}
\address[Mart\'{\i}n]{Departamento de An\'{a}lisis Matem\'{a}tico,
Facultad de Ciencias,
Universidad de Granada,
E-18071 Granada, Spain}
\email{\texttt{mmartins@ugr.es}}



\thanks{The second author was supported by Basic Science Research Program through the National Research Foundation of Korea(NRF) funded by the Ministry of Education, Science and Technology (NRF-2012R1A1A1006869).
The third author was partially supported by Spanish MICINN and FEDER project no.~MTM2012-31755, and by Junta de Andaluc\'{\i}a and FEDER grants FQM-185 and P09-FQM-4911.}

\begin{abstract}
We study the Bishop-Phelps-Bollob\'as property for numerical radius (in short, BPBp-$\nuu$) and find sufficient conditions for Banach spaces to ensure the BPBp-$\nuu$. Among other results, we show that $L_1(\mu)$-spaces have this property for every measure $\mu$. On the other hand, we show that every infinite-dimensional separable Banach space can be renormed to fail the BPBp-$\nuu$. In particular, this shows that the Radon-Nikod\'{y}m property (even reflexivity) is not enough to get BPBp-$\nuu$.
\end{abstract}

\date{December 29th, 2013. Revised February 14th, 2014.}

\maketitle

\section{Introduction}

Let $X$ be a (real or complex) Banach space and $X^*$ be its dual space. The unit sphere of $X$ will be denoted by $S_X$. We write $\mathcal{L}(X)$ for the space of all bounded linear operators on $X$. For $T\in \mathcal{L}(X)$, its {\it numerical radius} is defined by
\[
v(T) =\sup \{ |x^*Tx| : (x, x^*) \in \Pi(X)\},
\]
where $\Pi(X) = \{ (x, x^*) \in S_X \times S_{X^*} : x^*(x)=1\}$. It is clear that $v$ is a semi-norm on $\mathcal{L}(X)$. We refer the reader to the monographs \cite{B-D1,B-D2} for background.
An operator $T\in \mathcal{L}(X)$ {\it attains its numerical radius} if there exists $(x_0, x_0^*)\in \Pi(X)$ such that $v(T) = |x_0^*Tx_0|$.

In this paper we will discuss on the density of numerical radius attaining operators, actually on an stronger property called Bishop-Phelps-Bollob\'{a}s property for numerical radius. Let us present first a short account on the known results about numerical radius attaining operators. Motivated by the study of norm attaining operators initiated by J. Lindenstrauss in the 1960's,  B.~Sims \cite{Sims} asked in 1972 whether the numerical radius attaining operators are dense in the space of all bounded linear operators on a Banach space. I.~Berg and B.~Sims \cite{BS1} gave a positive answer for uniformly convex spaces and C.~Cardassi showed that the answer is positive for $\ell_1$, $c_0$,  $C(K)$ (where $K$ is a metrizable compact), $L_1(\mu)$ and uniformly smooth spaces \cite{Card1, Card2, Card3}. M.~Acosta showed that the numerical radius attaining operators are dense in $C(K)$ for every compact Hausdorff space $K$ \cite{Acosta1}. M.~Acosta and R.~Pay\'{a} showed that numerical radius attaining operators are dense in $\mathcal{L}(X)$ if $X$ has the Radon-Nikod\'ym property \cite{AP}. On the other hand, R.~Pay\'{a} \cite{Pay} showed in 1992 that there is a Banach space $X$ such that the numerical radius attaining operators are not dense in $\mathcal{L}(X)$, which gave a negative answer to Sim's question. Some authors also paid attention to the study of  denseness of numerical radius attaining nonlinear mappings \cite{CK, ABR, AK, KL}.

Motivated by the work \cite{AAGM2} of M.~Acosta, R.~Aron, D.~Garc\'{\i}a and M.~Maestre on the Bishop-Phelps-Bollob\'as property for operators, A.~Guirao and O.~Kozhushkina \cite{GuiOle} introduced very recently the notion of Bishop-Phelps-Bollob\'as property for numerical radius.

\begin{definition}[\textrm{\cite{GuiOle}}]
A Banach space $X$ is said to have the \emph{Bishop-Phelps-Bollob\'as property for numerical radius} (in short, \emph{BPBp-$\nuu$}) if for every $0<\eps<1$, there exists $\eta(\eps)>0$ such that
whenever $T\in \mathcal{L}(X)$ and $(x, x^*)\in \Pi(X)$ satisfy  $v(T)=1$ and $|x^*Tx|>1-\eta(\eps)$, there exit $S\in \mathcal{L}(X)$ and $(y, y^*)\in \Pi(X)$ such that
\[
v(S) = |y^*Sy|=1, \ \ \ \|T-S\|<\eps, \ \ \ \|x-y\|<\eps,\ \ \text{and} \ \ \  \|x^*- y^*\|<\eps.
\]
\end{definition}

Notice that if a Banach space $X$ has the BPBp-$\nuu$, then the numerical radius attaining operators are dense in $\mathcal{L}(X)$. One of the main results of this paper is to show that the converse result is not longer true (section~\ref{sec:counterexamples})

It is shown in \cite{GuiOle} that the real or complex spaces $c_0$ and $\ell_1$ have the BPBp-$\nuu$. This result has been extended to the real space $L_1(\R)$ by J.~Falc\'{o} \cite{Falco}. A.~Avil\'{e}s, A.~J.~Guirao and J.~Rodr\'{\i}guez \cite{AviGuiRod} give sufficient conditions on a compact space $K$ for the real space $C(K)$ to have the BPBp-$\nuu$ which, in particular, include all metrizable compact spaces.

The content of this paper is the following. First, we introduce in section~\ref{sec:modulus} a modulus of the BPBp-$\nuu$ analogous to the one introduced in \cite{ACKLM} for the Bishop-Phelps-Bollob\'{a}s property for the operator norm, and we will use it as a tool in the rest of the paper. As easy applications, we prove that finite-dimensional spaces always have the BPBp-$\nuu$ and that a reflexive space has the BPBp-$\nuu$ if and only if its dual does.  Next, section~\ref{sec:weak-BPBp-nu} is devoted to prove that Banach spaces which are both uniformly convex and uniformly smooth satisfy a weaker version of the BPBp-$\nuu$ and to discuss such weaker version. In particular, it is shown that $L_p(\mu)$ spaces have the BPBp-$\nuu$ for every measure $\mu$ when $1<p<\infty$, $p\neq 2$. We show in section~\ref{sec:L1} that given any measure $\mu$, the real or complex space $L_1(\mu)$ has the BPBp-$\nuu$. Finally, we prove in section~\ref{sec:counterexamples} that every separable infinite-dimensional Banach space can be equivalently renormed to fail the BPBp-$\nuu$ (actually, to fail the weaker version). In particular, this shows that reflexivity (or even superreflexivity) is not enough for the BPBp-$\nuu$, while the Radon-Nikod\'{y}m property was known to be sufficient for the density of numerical radius attaining operators.

Let us introduce some notations for later use. The $n$-dimensional space with the $\ell_1$ norm is denoted by $\ell_1^{(n)}$.  Given a family $\{X_k\}_{k=1}^\infty$ of Banach spaces,
$\big[\bigoplus_{k=1}^\infty X_k\big]_{c_0}$ (resp. $\big[\bigoplus_{k=1}^\infty X_k\big]_{\ell_1}$) is the Banach space consisting of all sequences $(x_k)_{k=1}^\infty$ such that each $x_k$ is in $X_k$ and $\lim\limits_{k\to \infty} \|x_k\|=0$ (resp. $\sum_{k=1}^\infty \|x_k\|<\infty$) equipped with the norm $\|(x_k)_{k=1}^\infty\| = \sup_{k} \|x_k\|$ (resp.  $\|(x_k)_{k=1}^\infty\| =\sum_{k=1}^\infty \|x_k\|$).

\section{Modulus of the Bishop-Phelps-Bollob\'{a}s for numerical radius}\label{sec:modulus}

Analogously to what is done in \cite{ACKLM} for the BPBp for the operator norm, we introduce here a modulus to quantify the Bishop-Phelps-Bollob\'{a}s property for numerical radius.

\begin{notation}
Let $X$ be a Banach space. Consider the set
\[
\Pi_{\nuu}(X) = \bigl\{ (x, x^*, T) \,:\, (x, x^*)\in \Pi(X),\, T\in \mathcal{L}(X),\, v(T)=1=|x^*Tx|\bigr\},
\]
which is closed in $S_X\times S_{X^*}\times \mathcal{L}(X)$ with respect to the following metric
\[
\dist \bigl((x,x^*, T), (y,y^*, S)\bigr) = \max\bigl\{   \|x-y\|, \|x^*-y^*\|, \|T-S\|\bigr\}.
\]
The \emph{modulus of the Bishop-Phelps-Bollob\'{a}s property for numerical radius} is the function defined by
\[
\eta_{\nuu}(X)(\eps) = \inf \Bigl\{ 1-|x^*Tx| \,:\, (x,x^*)\in \Pi(X),\, T\in \mathcal{L}(X),\, v(T)=1,\, \dist\bigl((x,x^*, T), \Pi_{\nuu}(X)\bigr)\geq \eps\Bigr\}
\]
for every $\eps\in (0,1)$. Equivalently, $\eta_{\nuu}(X)(\eps)$ is the supremum of those $\eta>0$ such that whenever $T\in \mathcal{L}(X)$ and $(x,x^*)\in \Pi(X)$ satisfy
$v(T)=1$ and $|x^*Tx|>1-\eta$, there exist $S\in \mathcal{L}(X)$ and $(y,y^*)\in \Pi(X)$ such that
\[
v(S) = |y^*Sy|=1, \ \ \ \|T-S\|<\eps, \ \ \ \|x-y\|<\eps,\ \ \text{and} \ \ \  \|x^*- y^*\|<\eps.
\]
\end{notation}

It is immediate that a Banach space $X$ has the BPBp-$\nuu$ if and only if $\eta_{\nuu}(\eps)>0$ for every $0<\eps<1$. By construction, if a function $\eps \longmapsto \eta(\eps)$ is valid in the definition of the BPBp-$\nuu$, then $\eta_{\nuu}(\eps)\geq \eta(\eps)$.

An immediate consequence of the compactness of the unit ball of a finite-dimensional space is the following result. It was previously known to A.~Guirao (private communication).

\begin{proposition}
Let $X$ be a finite dimensional Banach space. Then $X$ has the Bishop-Phelps-Bollob\'as property for numerical radius.
\end{proposition}

\begin{proof}
Let $K=\{ S\in \mathcal{L}(X) : v(S)=0\}$. Then $K$ is a norm-closed subspace of $\mathcal{L}(X)$. Hence $\mathcal{L}(X)/K$ is a finite-dimensional space with two norms
\begin{align*} v([T]) &:=\inf\{ v(T-S) : S\in K \} = v(T)\\
\|[T]\| &:= \inf\{ \| T-S \| : S\in K\},\end{align*}
 where $[T]$ is the class of $T$ in the quotient space $\mathcal{L}(X)/K$. Hence there is a constant $0<c\leq 1$ such that $c\|[T]\| \leq v(T)\leq \|[T]\|$.

Suppose that $X$ does not have the BPBp-$\nuu$. Then, there is $0<\eps<1$ such that $\eta_{\nuu}(X)(\eps)=0$. That is, there are  sequences $(x_n, x_n^*)\in \Pi(X)$ and $(T_n)\in \mathcal{L}(X)$ with $v(T_n)=1$ such that
\[
\dist\bigl((x_n, x_n^*, T_n), \Pi_{\nuu}(X)\bigr) \geq \eps \ \ (n\in \N)\qquad \text{and} \qquad \lim_{n} \bigl|x^*_n T_n x_n\bigr|=1.
\]
By compactness, we may assume that $\lim_{n} \|[T_n ] - [T_0] \|=0$ for some $T_0\in \mathcal{L}(X)$ and $v(T_0)=1$. Hence there exists a sequence $\{S_n\}_n$ in  $K$ such that $\lim_n \| T_n - (T_0 + S_n)\|=0$. Observe that $v(T_0+S_n)=v(T_0)=1$ for every $n\in \N$.

By compactness again, we may assume that $(x_n, x_n^*)$ converges to $(x_0, x^*_0)\in X\times X^*$. This implies that $(x_0,x_0^*)\in \Pi(X)$, and
$\bigl|x_0^*(T_0+S_n)x_0\bigr|=v(T_0+S_n)=1$, that is, $(x_0, x_0^*, T_0+S_n)\in \Pi_{\nuu}(X)$ for all $n$. This is a contradiction with the fact that
\begin{equation*}
0= \lim_{n}\dist\bigl( (x_n, x_n^*, T_n), (x_0, x_0^*, T_0+S_n) \bigr)\geq \lim_{n} \dist\bigl((x_n, x_n^*, T_n), \Pi_{\nuu}(X)\bigr) \geq \eps.\qedhere
\end{equation*}
\end{proof}

We may also give the following easy result concerning duality.

\begin{proposition}
Let $X$ be a reflexive space. Then $\eta_{\nuu}(X)(\eps)=\eta_{\nuu}(X^*)(\eps)$ for every $\eps\in (0,1)$. In particular, $X$ has the BPBp-$\nuu$ if and only if $X^*$ has the BPBp-$\nuu$.
\end{proposition}

We will use that $v(T^*) = v(T)$ for all $T\in \mathcal{L}(X)$, where $T^*$ denotes the adjoint operator of $T$. This result can be found in \cite{B-D1}, but it is obvious if $X$ is reflexive.

\begin{proof}
By reflexivity, it is enough to show that $\eta_{\nuu}(X)(\eps)\leq \eta_{\nuu}(X^*)(\eps)$. Let $\eps\in (0,1)$ be fixed. If $\eta_{\nuu}(X)(\varepsilon)=0$, there is nothing to prove. Otherwise, consider $0<\eta<\eta_{\nuu}(X)(\varepsilon)$. Suppose that $T_1\in \mathcal{L}(X^*)$ and $(x^*_1,x_1)\in \Pi(X^*)$ satisfy
$$
v(T_1)=1 \quad \text{and} \quad |x_1T_1x^*_1|>v(T_1)-\eta.
$$
By considering $T_1^*\in \mathcal{L}(X)$, we may find $S_1\in \mathcal{L}(X)$ and $(y_1,y_1^*)\in \Pi(X)$ such that
$$
|y_1^*S_1y_1|=v(S_1)=1,\quad \|y_1-x_1\|<\varepsilon,\quad \|y^*_1-x^*_1\|<\varepsilon\quad \text{and} \quad \|T_1^*-S_1\|<\varepsilon.
$$
Then $S_1^*\in \mathcal{L}(X^*)$ and $(y_1^*,y_1)\in \Pi(X^*)$ satisfy
$$
|\inner{y_1,S_1^*y_1^*}| =v(S_1)=1,\quad \|y_1^*-x_1^*\|<\varepsilon,\quad \|y_1-x_1\|<\varepsilon\quad \text{and} \quad \|T_1-S_1^*\|<\varepsilon.
$$
This implies that $\eta_{\nuu}(X^*)(\eps)\geq \eta$. We finish by just taking supremum on $\eta$.
\end{proof}

We do not know whether the result above is valid in the non-reflexive case.

\section{Spaces which are both uniformly convex and uniformly smooth}\label{sec:weak-BPBp-nu}

For a Banach space which is both uniformly convex and uniformly smooth, we get a property which is weaker than BPBp-$\nuu$. This result was known to A. Guirao (private communication).

\begin{proposition}\label{unif-convex-smooth}
Let $X$ be a uniformly convex and uniformly smooth Banach space. Then, given $\eps>0$, there exists $\eta(\eps)>0$ such that whenever $T_0\in \mathcal{L}(X)$ with $v(T_0)=1$ and $(x_0, x^*_0)\in \Pi(X)$ satisfy
$|x^*_0T_0x_0|>1-\eta(\eps)$, there exist $S\in \mathcal{L}(X)$ and $(y, y^*)\in \Pi(X)$ such that
\[
v(S) = |y^*Sy|, \quad \|x-y\|<\eps, \quad \|x^*-y^*\|<\eps \quad \text{and} \quad \|S-T_0\|<\eps.
\]
\end{proposition}

\begin{proof}
Notice that the uniform smoothness of $X$ is equivalent to the uniform convexity of $X^*$. Let $\delta_X(\eps)$ and $\delta_{X^*}(\eps)$ be the moduli of convexity $X$ and $X^*$, respectively. Given $0<\eps<1$, consider
$$
\eta(\eps) = \frac{\eps}{4}\min\left\{\delta_X\left(\frac{\eps}{4}\right), \delta_{X^*}\left(\frac{\eps}{4}\right)\right\}>0.
$$
Consider
$T_0\in \mathcal{L}(X)$ with $v(T_0)=1$ and $(x_0, x^*_0)\in \Pi(X)$ satisfying $|x^*_0T_0x_0|>1-\eta(\eps)$. Define $T_1\in \mathcal{L}(X)$ by $T_1x = T_0x +\lambda_1 \frac{\eps}{4}x_0^*(x)x_0$ for all $x\in X$, where $\lambda_1$ is the scalar satisfying $|\lambda_1|=1$ and $|x_0^*T_0x_0 + \lambda_1 \frac{\eps}4 | = |x_0^*T_0x_0| + \frac{\eps}4$.
Now, choose $x_1\in S_X$ and $x_1^*\in S_{X^*}$ such that  $|x_1^*(x_1)|=1$, $x_1^*(x_0) = |x_1^*(x_0)|$ and
\[ |x_1^*T_1x_1| \geq v(T_1) - \eta(\frac{\eps^2}{4^2}).\]
Now we define a sequence $(x_n, x_n^*, T_n)$ in
$S_X\times S_{X^*}\times \mathcal{L}(X)$ inductively. Indeed, suppose that we have a defined sequence $(x_j, x_j^*, T_j)$ for $0\leq j\leq n$ and let
\[ T_{n+1} x = T_n x + \lambda_{n+1} \frac{\eps^{n+1}}{4^{n+1}} x_n^*(x)x_n.\] Then choose $x_{n+1} \in S_X$ and $x_{n+1}^*\in S_{X^*}$ such that $|x_{n+1}^*(x_{n+1})| = 1$ and $|x_{n+1}^*(x_n)| = x_{n+1}^*(x_n)$
\[
\bigl|x_{n+1}^*T_{n+1} x_{n+1}\bigr| \geq v(T_{n+1}) - \eta\left(\frac{\eps^{n+2}}{4^{n+2}}\right).
\]
Notice that for all $n\geq 0$, we have
\[ \|T_{n+1} - T_n\| \leq \frac{\eps^{n+1}}{4^{n+1}} \ \ \ \text{and} \ \ \ |v(T_{n+1}) -v(T_n) | \leq\frac{\eps^{n+1}}{4^{n+1}}.\] This implies that $(T_n)$ is a Cauchy sequence and assume that it converges to $S\in \mathcal{L}(X)$. Then we have
\[\lim_n T_n = S,~~ \|T_0 - S\|<\eps~~\text{and}~~\lim_n |x_n^*T_n x_n | = \lim_n v(T_n) = v(S).\]
We will show that both sequences $(x_n)$ and $(x_n^*)$ are Cauchy.
From the definition, we have
\begin{align*}
 v(T_{n+1}) - \eta\left(\frac{\eps^{n+2}}{4^{n+2}}\right) &\leq |x^*_{n+1}T_{n+1}x_{n+1}|\\
 &\leq \left|x_{n+1}^* T_n x_{n+1} + \lambda_{n+1}\frac{\eps^{n+1}}{4^{n+1}} x_n^*(x_{n+1})x_{n+1}^*(x_n)\right|\\
 &\leq v(T_n) + \frac{\eps^{n+1}}{4^{n+1}}x_{n+1}^*(x_n)
\end{align*}
and
\begin{align*}
v(T_{n+1}) &\geq |x_n^*T_{n+1}x_n| = \left|x_n^*T_nx_n + \lambda_{n+1}\frac{\eps^{n+1}}{4^{n+1}}\right|\\
 & =  |x_n^*T_nx_n | + \frac{\eps^{n+1}}{4^{n+1}} \geq v(T_n) - \eta\left(\frac{\eps^{n+1}}{4^{n+1}}\right) + \frac{\eps^{n+1}}{4^{n+1}}.
\end{align*}
In summary,  we have
\[ v(T_n) + \frac{\eps^{n+1}}{4^{n+1}}x_{n+1}^*(x_n) \geq  v(T_n) - \eta\left(\frac{\eps^{n+1}}{4^{n+1}}\right) + \frac{\eps^{n+1}}{4^{n+1}}- \eta\left(\frac{\eps^{n+2}}{4^{n+2}}\right).\]
Hence
\[
x_{n+1}^*(x_n) \geq 1- 2 \frac{4^{n+1}}{\eps^{n+1}}\eta\left(\frac{\eps^{n+1}}{4^{n+1}}\right) = 1-\frac12\min\left\{\delta_X\left( \frac{\eps^{n+1}}{4^{n+2}}\right), \delta_{X^*}\left(\frac{\eps^{n+1}}{4^{n+2}}\right)\right\}
\]
and
\begin{align*}
\left\|\frac{x_n+x_{n+1}}2\right\|&\geq  x_{n+1}^*\left(\frac{x_n+x_{n+1}}2\right)\geq 1- \delta_X\left( \frac{\eps^{n+1}}{4^{n+2}}\right),
\\ \left\|\frac{x_n^*+x_{n+1}^*}2\right\|&\geq  \frac{x_n^*+x_{n+1}^*}2(x_n)\geq 1- \delta_{X^*}\left(\frac{\eps^{n+1}}{4^{n+2}}\right).
\end{align*}
This means that $\|x_n - x_{n+1} \| \leq \frac{\eps^{n+1}}{4^{n+2}}$  and  $\|x_n^* - x_{n+1}^* \| \leq \frac{\eps^{n+1}}{4^{n+2}}$ for all $n$. So $(x_n)$  and $(x_n^*)$ are Cauchy. Let $x_\infty= \lim_n x_n$ and $x^*_\infty = \lim_n x_n^*$. Then we have $\|x_0-x_\infty\|<\frac{\eps}4$ and $\|x^*_0-x^*_\infty\|<\frac{\eps}4$. Hence, $|x_\infty^*(x_\infty)| = \lim_n |x_n^*(x_n)|=1$ and
\[ v(S) = \lim_n v(T_n) = \lim_n |x_n^*T_n x_n| =  |x^*_\infty Sx_\infty|.\]
Let $\alpha = x_\infty^*(x_\infty)$, $y^* = \bar{\alpha}x^*_\infty$ and $y=x_\infty$. Then we have $y^*(y)=1$, $v(S) = |y^*Sy|$ and $\|y-x_0\|<\eps$. Notice that
\[ |\alpha -1 | = |x_\infty^*(x_\infty) - x_0^*(x_0)| \leq |(x_\infty^*-x_0^*)(x_\infty)| + |x_0^*(x_\infty) - x_0^*(x_0)| <\frac{\eps}2.
\]
Therefore
\[\|y^* - x^* \| \leq \|\bar\alpha y^* - y^*\| + \|y^*-x^*\| < \frac{\eps}2+\frac{\eps}4<\eps.\]
This completes the proof.
\end{proof}

Let us discuss a little bit about the equivalence between the property in the result above and the BPBp-$\nuu$. For convenience, let us introduce the following definition.

\begin{definition}
A Banach space $X$ has the \emph{weak Bishop-Phelps-Bollob\'{a}s property for the numerical radius} (in short \emph{weak-BPBp-$\nuu$}) if given $\eps>0$, there exists $\eta(\eps)>0$ such that whenever $T_0\in \mathcal{L}(X)$ with $v(T_0)=1$ and $(x_0, x^*_0)\in \Pi(X)$ satisfy
$|x^*_0T_0x_0|>1-\eta(\eps)$, there exist $S\in \mathcal{L}(X)$ and $(y, y^*)\in \Pi(X)$ such that
\[
v(S) = |y^*Sy|, \quad \|x-y\|<\eps, \quad \|x^*-y^*\|<\eps \quad \text{and} \quad \|S-T\|<\eps.
\]
\end{definition}

Notice that the only difference between this concept and the BPBp-$\nuu$ is the normalization of the operator $S$ by the numerical radius. Of course, if the numerical radius and the operator norm are equivalent, this two properties are the same. This equivalence is measured by the so-called numerical index of the Banach space, as follows.
For a Banach space $X$, the {\it numerical index} of $X$ is defined by
\[
n(X) = \inf \{ v(T) \,:\, T\in \mathcal{L}(X), \|T\|=1\}.
\]
It is clear that $0\leq n(X)\leq 1$ and $n(X)\|T\|\leq v(T) \leq \|T\|$ for all $T\in \mathcal{L}(X)$. The value $n(X)=1$ means that $v$ equals the usual operator norm. This is the case of $X=L_1(\mu)$ and $X=C(K)$, among many others. On the other hand, $n(X)>0$ if and only if the  numerical radius is equivalent to the norm of $\mathcal{L}(X)$. We refer the reader to \cite{KMP} for more information and background.

The following result is immediate. We include a proof for the sake of completeness.

\begin{proposition}\label{prop-nxpositive}
Let $X$ be a Banach space with $n(X)>0$. Then, $X$ has the BPBp-$\nuu$ if and only if $X$ has the weak-BPBp-$\nuu$.
\end{proposition}

\begin{proof}
The necessity is clear. For the converse, assume that we have $\eta(\eps)>0$ satisfying the conditions of the weak-BPBp-$\nuu$ for all $0<\eps<1$.
If $T\in \mathcal{L}(X)$ with $v(T)=1$ and $(x_0, x^*_0)\in \Pi(X)$ satisfy
$|x^*_0Tx_0|>1-\eta(\eps)$ for $0<\eps<1$, then there exist $S\in \mathcal{L}(X)$ and $(y, y^*)\in \Pi(X)$ such that
\[ v(S) = |y^*Sy|, \ \ \  \|S-T\|<\eps, \ \ \ \|x-y\|<\eps~~~ \text{and}~~~\|x^*-y^*\|<\eps.\]
As $v(S)>0$ by the above, let $S_1 = \frac{1}{v(S)}S$. Then we have
\[  1=v(S_1)= |y^*S_1 y|,~~~~\|x-y\|<\eps~~~ \text{and}~~~\|x^*-y^*\|<\eps.\]
Finally, we have
\begin{align*}
 \|S_1 - T\| &\leq \left\|\frac{1}{v(S)} S - S\right\| + \|S-T\| = \frac{\|S\|}{v(S)} |v(S) - 1| + \|S-T\|\\
&\leq\frac{1}{n(X)} |v(S)- v(T)| + \|S-T\|\\
&\leq \left( \frac{1}{n(X)}+1\right) \|S-T\| < \frac{n(X)+1}{n(X)} \eps.
\end{align*}
An obvious change of parameters finishes the proof.
\end{proof}

We do not know whether the hypothesis of $n(X)>0$ can be omitted in the above result.

Putting together Propositions \ref{unif-convex-smooth} and \ref{prop-nxpositive}, we get the following.

\begin{corollary}
Let $X$ be a uniformly convex and uniformly smooth Banach space with $n(X)>0$. Then $X$ has the BPBp-$\nuu$.
\end{corollary}

Let us comment that every complex Banach space $X$ satisfies $n(X)\geq 1/\e$, so the above corollary automatically applies in the complex case. In the real case, this is not longer true, as the numerical index of a Hilbert space of dimension greater than or equal to two is $0$. On the other hand, it is proved in \cite{MMP-Israel} that real $L_p(\mu)$ spaces have non-zero numerical index for every measure $\mu$ when $p\neq 2$. Therefore, we have the following examples.

\begin{examples}$ $
\begin{enumerate}
\item[(a)] Complex Banach spaces which are uniformly smooth and uniformly convex satisfy the BPBp-$\nuu$.
\item[(b)] In particular, for every measure $\mu$, the complex spaces $L_p(\mu)$ have the BPBp-$\nuu$ for $1<p<\infty$.
\item[(c)] For every measure $\mu$, the real spaces $L_p(\mu)$ have the BPBp-$\nuu$ for $1<p<\infty$, $p\neq 2$.
\end{enumerate}
\end{examples}

\textbf{Note added in revision:}  Very recently, H.~J.~Lee, M.~Mart\'in and J.~Mer\'i have proved that Proposition~\ref{prop-nxpositive} can be extended to some Banach spaces with numerical index zero as, for instance, real Hilbert spaces. Hence, they have shown that Hilbert spaces have the BPBp-$\nuu$. These results will appear elsewhere.

\section{$L_1$ spaces}\label{sec:L1}

In this section, we will show that $L_1(\mu)$ has the BPBp-$\nuu$ for every measure $\mu$. In the proof, we are dealing with complex integrable functions since the real case is followed easily by applying the same proof. Our main result here is the following.

\begin{theorem}\label{thr:L_1-arbitrary}
Let $\mu$ be a measure. Then $L_1(\mu)$ has the Bishop-Phelps-Bollob\'{a}s property for numerical radius. More precisely, given $\eps>0$, there exists $\eta(\eps)>0$ (which does not depend on $\mu$) such that whenever $T_0\in\mathcal{L}(L_1(\mu))$ with $v(T_0)=1$ and $(f_0, g_0) \in \Pi(L_1(\mu))$ satisfy $|\inner{T_0f_0, g_0}|>1-\eta(\eps)$, then there exist $T\in \mathcal{L}(L_1(\mu))$, $(f_1, g_1) \in \Pi(L_1(\mu))$ such that
\[ |\inner{Tf_1, g_1}|=v(T)=1,\ \  \|f_0-f_1\|< \eps, \ \ \|g_0-g_1\|< \eps\ \ \text{and}\ \ \|T-T_0\|< \eps.\]
\end{theorem}

As a first step, we have to start dealing with finite regular positive Borel measures, for which a representation theorem for operators exists.

\begin{prop}\label{maintheorem}
Let $m$ be a finite regular positive Borel measure on a compact Hausdorff space $\Omega$. Then $L_1(m)$ has the Bishop-Phelps-Bollob\'as property for numerical radius. More precisely, given $\eps>0$, there is $\eta(\eps)>0$ (which is independent of the measure $m$) such that
if $T$ is a norm-one element in $\mathcal{L}(L_1(m))$ and there exists an $(f_0, g_0) \in \Pi(L_1(m))$ satisfying $|\inner{Tf_0, g_0}|>1-\eta(\eps)$ , then there exist an operator $S\in \mathcal{L}(L_1(m))$, $(f_1, g_1) \in \Pi(L_1(m))$ such that
\[ |\inner{Sf_1, g_1}|=\|S\|=1,\ \  \|f_0-f_1\|\leq \eps, \ \ \|g_0-g_1\|\leq \eps\ \ \text{and}\ \ \|T-S\|\leq \eps.\]
\end{prop}

To prove this proposition, we need some background on representation of operators on Lebesgue spaces on finite regular positive Borel measures and several preliminary lemmas.

Let $m$ be a finite regular positive Borel measure on a compact Hausdorff space $\Omega$. If $\mu$ is a complex-valued Borel measure on the product space $\Omega\times \Omega$, then define their marginal measures $\mu^i$ on $\Omega$ ($i=1,2$) as following:
$\mu^1(A) = \mu(A\times \Omega)$ and $\mu^2(B)= \mu(\Omega\times B)$, where $A$ and $B$ are Borel measurable subsets of $\Omega$.

Let $M(m)$ be the complex Banach lattice of measures consisting of all complex-valued Borel measures $\mu$ on the product space $\Omega\times \Omega$ such that $|\mu|^i$ are absolutely continuous with respect to $m$ for $i=1,2$, endowed with the norm
\[\nor{ \frac{d|\mu|^1}{dm}}_\infty.\]
Each $\mu\in M(m)$ defines a bounded linear operator $T_\mu$ from $L_1(m)$ to itself by
\[
\inner{T_\mu (f), g}  = \int_{\Omega\times \Omega} f(x)g(y) \, d\mu(x,y),
\]
where $f\in L_1(m)$ and $g\in L_\infty(m)$. A.~Iwanik \cite{Iwanik} showed that the mapping $\mu\longmapsto T_\mu$ is a lattice isometric isomorphism from $M(m)$ onto $\mathcal{L}(L_1(m))$. Even though he showed this for the real case, it can be easily generalized to the complex case. For details, see \cite[Theorem~1]{Iwanik} and  \cite[IV Theorem 1.5 (ii), Corollary 2]{Schaefer}.

We will also use that given an arbitrary measure $\mu$, every $T\in \mathcal{L}(L_1(\mu))$ satisfies $v(T)=\|T\|$ \cite{DMPW} (that is, the space $L_1(\mu)$ has numerical index $1$).

\begin{lemma}[\textrm{\cite[Lemma~3.3]{AAGM2}}]\label{elementary}
Let $\{c_n\}$ be a sequence of complex numbers with $|c_n|\leq 1$ for every $n$, and let $\eta>0$ be such that for a convex series $\sum \alpha_n$, $\re \sum_{n=1}^\infty \alpha_n c_n >1-\eta$. Then for every $0<r<1$, the set $A : = \{ i \in \mathbb{N} : \re c_i > r \}$, satisfies the estimate
\[ \sum_{i\in A} \alpha_i \geq 1-\frac{\eta}{1-r}.\]
\end{lemma}

From now on, $m$ will be a finite regular positive Borel measure on the compact Hausdorff space $\Omega$.

\begin{lemma}\label{lem:real1}
Suppose that there exist a non-negative simple function $f\in S_{L_1(m)}$ and a function $g\in S_{L_\infty(m)}$ such that
\[\re \inner{f, g} >1-\frac{\eps^3}{16}.\] Then there exist a nonnegative simple function $f_1\in S_{L_1(m)}$ and a function $g_1\in S_{L_\infty(m)}$ such that
\[g_1(x) =\chi_{{\rm supp }(f_1)}(x) + g(x)\chi_{\Omega\setminus {\rm supp}(f_1)}(x),\]
\[ \inner{f_1, g_1} =1,\ \  \|f-f_1\|_1<\eps, \ \ \|g-g_1\|_\infty<\sqrt\eps  \ \ \text{and} \ \ \ \ {\rm supp}( f_1) \subset {\rm supp}(f).\]
\end{lemma}

\begin{proof}
Let$f= \sum_{j=1}^m \frac{\beta_j}{m(B_j)} \chi_{B_j}$ for some $(\beta_j)$ such that $\beta_j\geq 0$ for all $j$ and $\sum_{j=1}^m\beta_j=1$, and $B_j$'s are mutually disjoint. By the assumption, we have
\[ \re \inner{f, g} =\sum_{j=1}^n \beta_j \frac{1}{m(B_j)} \int_{B_j}\re g(x) \, dm(x)>1-\frac{\eps^3}{16},\] and letting
\[ J = \{ j : 1\leq j\leq n,  \frac{1}{m(B_j)} \int_{B_j}\re  g(x) \, dm(x)>1-\frac{\eps^2}4\} ,\] we have by Lemma~\ref{elementary}
\[ \sum_{j\in J } \beta_j >1-\frac{\eps}4.\]
For each $j\in J$, we have
\begin{align*}
1-\frac{\eps^2}4 &<  \frac{1}{m(B_j)} \int_{B_j} \re g(x) \, dm(x) \\
&= \frac{1}{m(B_j)} \int_{B_j\cap \{\re g\leq 1-\eps \}}\re g(x)\, dm(x)+  \int_{B_j\cap \{\re g>1-\eps\} } \re g(x) \, dm(x)\\
&\leq \frac{1}{m(B_j)} \left( (1-\eps)m( B_j\cap \{ \re g\leq 1-\eps \}) + m(B_j\cap \{\re g>1-\eps\} )  \right)\\
&=1-\eps \frac{m(B_j \cap \{ \re g\leq 1-\eps\}) }{m(B_j)}.
\end{align*}
This implies that
\[
\frac{m(B_j \cap \{\re g\leq 1-\eps\}) }{m(B_j)} < \frac{\eps}4.
\]
Define $\tilde B_j = B_j \cap \{ \re g> 1-\eps\}$ for all $j\in J$, $f_1 = \frac{1}{\sum_{j\in J} \beta_j}\sum_{j\in J} \beta_j \frac{\chi_{\tilde B_j}}{m(\tilde B_j)}$ and $g_1(x) = 1 $ on ${\rm supp}(f_1)$ and $g_1(x) = g(x)$ elsewhere. Then it is clear that ${\rm supp }(f_1) \subset {\rm supp}(f)$,  $\|g-g_1\|_\infty<\sqrt\eps$ and $\inner{f_1, g_1}=1$.
Finally we will show that $\|f-f_1\|<\eps$. Notice first that
\begin{align*}
\nor{\sum_{j\in J} \beta_j\frac{\chi_{\tilde B_j}}{m(\tilde B_j)} - \sum_{j\in J}\beta_j \frac{\chi_{B_j}}{m(B_j)}   } &\leq  \nor{\sum_{j\in J}\beta_j \frac{\chi_{\tilde B_j}}{m(\tilde B_j)} - \sum_{j\in J}\beta_j \frac{\chi_{\tilde B_j}}{m(B_j)}   } + \nor{\sum_{j\in J} \beta_j\frac{\chi_{\tilde B_j}}{m(B_j)} - \sum_{j\in J}\beta_j \frac{\chi_{B_j}}{m(B_j)}   } \\
&=2 \sum_{j\in J} \beta_j \frac{m(B_j\setminus \tilde B_j)}{m(B_j)} <\frac{\eps}2.
\end{align*}
Hence
\begin{align*}
\|f-f_1\| &\leq \nor{ \frac{1}{\sum_{j\in J} \beta_j}   \sum_{j\in J} \beta_j \frac{\chi_{\tilde B_j}}{m(\tilde B_j)}  -   \sum_{j\in J} \beta_j \frac{\chi_{\tilde B_j}}{m(\tilde B_j)} }  + \nor{ \sum_{j\in J} \beta_j \frac{\chi_{\tilde B_j}}{m(\tilde B_j)}-f  }\\
&\leq  \frac{1-\sum_{j\in J} \beta_j }{\sum_{j\in J} \beta_j } ~\nor{  \sum_{j\in J} \beta_j \frac{\chi_{\tilde B_j}}{m(\tilde B_j)}  }+
 \nor{ \sum_{j\in J} \beta_j \frac{\chi_{\tilde B_j}}{m(\tilde B_j)}-  \sum_{j\in J} \beta_j\frac{\chi_{B_j}}{m(B_j)}   }+\frac{\eps}{4}\\
 &=(1-\sum_{j\in J} \beta_j) +  \frac{\eps}2 + \frac{\eps}4\\
&\leq \frac{\eps}4 + \frac{\eps}2 + \frac{\eps}4= \eps.\qedhere
\end{align*}
\end{proof}

\begin{lemma}[\textrm{\cite[Lemma~3.3]{CKLM}}]\label{lem:real2}
Suppose that $T_\mu$ is a norm-one element in $\mathcal{L}(L_1(m))$ for some $\mu\in M(m)$ and there is a nonnegative simple function $f_0$ such that $f_0$ is a norm-one element of  $L_1(m)$ and $\|T_\mu f_0 \|\geq 1-\eps^3/2^6$ for some $0<\eps<1$. Then there exist  a norm-one bounded linear operator $T_\nu$ for some $\nu\in M(m, m)$ and a nonnegative simple function $f_1$ in $S_{L_1(m)}$ such that $\norm{T_\mu - T_\nu}\leq\eps$, $\|f_1- f_0\|\leq 3\eps$ and $\frac{d|\nu|^1}{dm}(x)=1$ for all $x\in {\rm supp}( f_1)$.
\end{lemma}

\begin{lemma}\label{lem:real3}
Suppose that $T_\nu\in \mathcal{L}(L_1(m))$ is a norm-one operator, $f=\sum_{i=1}^n \beta_i \frac{\chi_{B_i}}{m(B_i)}$, where $m(B_j)>0$ for all $1\leq j\leq n$ and  $\{B_j\}_{j=1}^n$ are mutually disjoint Borel subsets of $\Omega$, is a norm-one nonnegative simple function and $g$ is an element of $S_{L_\infty(m)}$ such that
\[\re \inner{g, T_\nu f} \geq 1-\frac{\eps^6}{2^7}\] for some $0<\varepsilon<1$ and
\[\frac{d|\nu|^1}{dm}(x)=1, \ \ \ \ g(x) = 1\] for all x in the support of $f$.

Then there exist a nonnegative simple function $\tilde{f}\in S_{L_1(m)}$,  a function $\tilde g\in S_{L_\infty(m)}$ and  an operator  $T_{\tilde\nu}$ in $\mathcal{L}(L_1(m), L_1(m))$  such that
\[\inner{\tilde g, T_{\tilde\nu}\tilde f}=\|T_{\tilde\nu}\|=1,\ \ \|T_\nu - T_{\tilde\nu}\|\leq 2\eps,\ \ \|f-\tilde f\|\leq 3\eps,\ \ \|g-\tilde g\|\leq \sqrt\eps\  \text{ and } \ \inner{\tilde f, \tilde g}=1.\]
\end{lemma}
\begin{proof}
Since \[\re \inner{g, T_\nu f} \geq 1-\frac{\eps^6}{2^7},\] we have
\begin{align*}
 1-\frac{\eps^6}{2^7}< \re  \inner{g, T_\nu f} = \int_{\Omega\times \Omega} f(x)\re g(y) \, d\mu(x,y) = \sum_{j=1}^n \beta_j \int_{\Omega\times \Omega} \frac{\chi_{B_j}(x)}{m(B_j)}\ \re g(y)\, d\nu(x,y).
 \end{align*}
Let $J= \{ j : \int_{\Omega\times \Omega} \frac{\chi_{B_j}(x)}{m(B_j)}\re g(y)\, d\nu(x,y) > 1-\frac{\eps^3}{2^6}\}$. Then from Lemma~\ref{elementary} we have $\sum_{j\in J}\beta_j >1-\frac{\eps^3}{2}$.
Let $f_1 = \sum_{j\in J} \tilde \beta_j \frac{\chi_{B_j}}{m(B_j)}$, where $\tilde \beta_j = \beta_j / (\sum_{j\in J}\beta_j )$ for all $j\in J$. Then
\[
\|f_1 - f\| \leq \nor{ \sum_{j\in J} (\tilde \beta_j - \beta_j) \frac{\chi_{B_j}}{m(B_j)} } + \sum_{j\in J} \beta_j \leq \eps^3\leq \eps.
\]
Note that there is a Borel measurable function $h$ on $\Omega\times \Omega$ such that $d\nu(x,y) = h(x,y)\, d|\nu|(x,y)$ and $|h(x,y)|=1$ for all $(x,y)\in \Omega\times \Omega$. Let
\[ C= \left\{ (x,y) \,:\, |g(y)h(x,y)-1| < \frac{\sqrt\eps}{2^{3/2}} \right\}.\]
Define  two measures $\nu_{f}$ and $\nu_{c}$ as follows:
\begin{equation*}
\nu_f(A) = \nu(A\setminus C) \quad \text{and} \quad
\nu_c(A) = \nu(A\cap C)
\end{equation*}
for every Borel subset $A$ of $\Omega\times \Omega$. It is clear that
$$
d\nu= d\nu_{f} + d\nu_{c},\quad d|\nu_f|  =\bar  hd\nu_f,\quad  d|\nu_c| =\bar  hd\nu_c, \quad \text{and}\quad  d|\nu| = d|\nu_f|+d|\nu_c|.
$$
Since $\frac{d|\nu|^1}{dm_1}(x)=1$ for all $x\in \bigcup_{j=1}^n B_j$, we have
\[
1=\frac{d|\nu|^1}{dm_1}(x)  = \frac{d|\nu_{f}|^1}{dm_1}(x)+ \frac{d|\nu_{c}|^1}{dm_1}(x)
\]
for all $x\in B=\bigcup_{j=1}^n B_j$,
and we deduce that $|\nu|^1(B_j)=m_1(B_j)$ for all $1\leq j\leq n$.

We claim that $\frac{|\nu_{f}|^1(B_j) }{m_1(B_j)} \leq \frac{\eps^2}{2^2}$ for all $j\in J$. Indeed, if $|g(y)h(x,y)-1|\geq  \frac{\sqrt\eps}{2^{3/2}}$, then $\re (g(y)h(x,y))\leq 1-\frac{\eps}{2^4}$. So we have
\begin{align*}
1-\frac{\eps^3}{2^6} &\leq \frac{1}{m_1(B_j)}\re \int_{\Omega\times \Omega} \chi_{B_j(x)} g(y)\, d\nu(x,y)\\
&= \frac{1}{m_1(B_j)} \int_{\Omega\times \Omega} \chi_{B_j(x)} \re\big(g(y)h(x,y)\big)  \, d|\nu|(x,y)\\
&= \frac{1}{m_1(B_j)} \int_{\Omega\times \Omega} \chi_{B_j(x)} \re\big(g(y)h(x,y)\big)  \, d|\nu_f|(x,y) \\
&\ \ \ \ \ +  \frac{1}{m_1(B_j)} \int_{\Omega\times \Omega} \chi_{B_j(x)} \re\big(g(y)h(x,y)\big)  \, d|\nu_c|(x,y) \\
&\leq \frac{1}{m_1(B_j)} \left((1-\frac{\eps}{2^4}) |\nu_f|^1(B_j) +|\nu_c|^1(B_j)\right) = 1 - \frac{\eps}{2^4} \frac{|\nu_{f}|^1(B_j)}{m_1(B_j)}.
\end{align*} This proves our claim.

We also claim that for each $j\in J$, there exists a Borel subset $\tilde B_j$ of $B_j$  such that
\[ \left(1-\frac{\eps}2\right) m_1(B_j) \leq m_1(\tilde B_j) \leq m_1(B_j)\] and
\[ \frac{d|\nu_{f}|^1}{dm_1}(x)\leq \frac{\eps}2 \] for all $x\in \tilde B_j$.
Indeed, set $\tilde B_j = B_j\cap \left\{ x \in \Omega : \frac{d|\nu_{f}|^1}{dm_1}(x)\leq\frac {\eps}2\right\}$. Then
\[ \int_{B_j \setminus \tilde B_j} \frac{\eps}2 \, dm_1(x) \leq  \int_{B_j} \frac{d|\nu_{f}|^1}{dm_1}(x) \, dm_1(x) = |\nu_{f}^1|(B_j)\leq \frac{\eps^2}{2^2}m_1(B_j). \] This shows that $m_1(B_j \setminus \tilde B_j) \leq \frac{\eps}2 m_1(B_j)$. This proves our second claim.

Now, we define $\tilde{g}$ by $\tilde g(y) = \frac{g(y)}{|g(y)|}$ if $|g(y)|\geq 1-\frac{\sqrt{\eps}}{2^{3/2}}$ and $\tilde g(y) = g(y)$ if $|g(y)|<1-\frac{\sqrt{\eps}}{2^{3/2}}$, and we write $\tilde f = \sum_{j\in J} \tilde\beta_j \frac{\chi_{\tilde B_j} }{m_1(\tilde B_j) }$.
It is clear that $\tilde g \in S_{L_\infty(m)}$, $\|g-\tilde g\|<\sqrt\eps$ and $\tilde g(y) = 1$ for all $x\in  {\rm supp} \tilde{f}$.

Finally, we define the measure
\[
d\tilde \nu(x,y) = \sum_{j\in J} \chi_{\tilde B_j}(x)  \overline{\tilde g(y)} \overline{ h(x, y)} d\nu_c(x,y) \left( \frac{d|\nu_{c}|^1}{dm_1}(x) \right)^{-1} + \chi_{J_1\setminus \tilde B}(x) d\nu(x,y),
\]
where $\tilde B= \bigcup_{j\in J} \tilde B_j$. It is easy to see that $\frac{d|\tilde\nu|^1}{dm_1}(x) =1$ on $\tilde B$ and $\frac{d|\tilde\nu|^1}{dm_1}(x)\leq 1$ elsewhere. Note that
\begin{align*}
d(\tilde \nu- \nu)(x,y) &= \sum_{j\in J} \chi_{\tilde B_j}(x) \left[ \overline{\tilde g(y)} \overline{ h(x, y)}\left( \frac{d|\nu_{c}|^1}{dm_1}(x) \right)^{-1} -1\right] d\nu_c(x,y)   \\
&\ \ \ \ - \sum_{j\in J} \chi_{\tilde B_j}(x) d\nu_f(x,y).
\end{align*}
If $(x,y)\in C$, then $|g(y)|\geq 1-\frac{\sqrt{\eps}}{2^{3/2}} \geq 1-\frac{1}{2^{3/2}}$ and
\begin{align*}
 \left| \overline{\tilde g(y)} \overline{ h(x, y)}-1 \right| &=  \left|  \frac{g(y)}{|g(y)|}  h(x, y)-1  \right| \\
 & \leq \frac{ \left|  g(y) h(x, y)-1  \right| }{|g(y)|} + \frac{\big|1-|g(y)|\big|}{|g(y)|} \\
 &\leq 2 \frac{ \left|  g(y) h(x, y)-1  \right| }{|g(y)|} \leq 2\frac{\sqrt{\eps}}{2^{3/2}}\frac{2^{3/2}}{2^{3/2}-1}\leq 2\sqrt{\eps}.
\end{align*}
Hence, for all $(x,y)\in C$ we have
\begin{align*}
 \left| \overline{\tilde g(y)} \overline{ h(x, y)}\left( \frac{d|\nu_{c}|^1}{dm_1}(x) \right)^{-1} -1\right|  &\leq
 \left| \overline{\tilde g(y)} \overline{ h(x, y)}-1 \right|~ \left( \frac{d|\nu_{c}|^1}{dm_1}(x) \right)^{-1}  +   \left| \left( \frac{d|\nu_{c}|^1}{dm_1}(x) \right)^{-1} -1 \right|\\
 &\leq   2\sqrt{\eps}~ \left( \frac{d|\nu_{c}|^1}{dm_1}(x) \right)^{-1} +\left| \left( \frac{d|\nu_{c}|^1}{dm_1}(x) \right)^{-1} -1 \right|.
\end{align*}
So, we have for all $x\in J_1$,
\begin{align*}
\frac{d| \tilde \nu - \nu|^1}{dm_1}(x) &\leq \sum_{j\in J} \chi_{\tilde B_j}(x) \left[  2\sqrt{\eps}~ \left( \frac{d|\nu_{c}|^1}{dm_1}(x) \right)^{-1} +\left| \left( \frac{d|\nu_{c}|^1}{dm_1}(x) \right)^{-1} -1 \right|
   \right] \frac{d|\nu_c|^1}{dm_1}(x) \\
&\ \ \ \ + \sum_{j\in J} \chi_{\tilde B_j}(x) \frac{d|\nu_f|^1}{dm_1}(x)\\
&\leq \sum_{j\in J} \chi_{\tilde B_j}(x) \left( 2\sqrt{\eps} +   \left( 1-\frac{d|\nu_{c}|^1}{dm_1}(x)   \right)\right) + \sum_{j\in J} \chi_{\tilde B_j}(x) \left( \frac{d|\nu_{f}|^1}{dm_1}(x) \right)\\
&\leq 2\sqrt{\eps}+\eps <3\sqrt{\eps}.
\end{align*}
This gives that $\|T_\nu - T_{\tilde \nu} \|< 3\sqrt\eps$. Note also that, for all $j\in J$,
\begin{align*}
 \inner{ T_{\tilde \nu} \frac{\chi_{\tilde B_j}}{m_1(\tilde B_j)}, \tilde g} &= \int_{\Omega\times \Omega} \frac{\chi_{\tilde B_j}(x)}{m_1(\tilde B_j)} \tilde  g(y) \,d\tilde \nu(x,y)\\
 & =\int_{\Omega\times \Omega} \frac{\chi_{\tilde B_j}(x)}{m_1(\tilde B_j)} \overline{h(x,y)}\left( \frac{d|\nu_{c}|^1}{dm_1}(x) \right)^{-1} \,\, d\nu_c(x,y)\\
 &=\int_{\Omega} \frac{\chi_{\tilde B_j}(x)}{m_1(\tilde B_j)}\left( \frac{d|\nu_{c}|^1}{dm_1}(x) \right)^{-1} \,  d|\nu_c|^1(x) \\
 &= \int_{\Omega} \frac{\chi_{\tilde B_j}(x)}{m_1(\tilde B_j)} \, dm_1(x)=1.
 \end{align*}
Hence we get $\inner{T_{\tilde \nu} \tilde f, \tilde g} =1$, which implies that $\|T_{\tilde \nu} \tilde f\|=\|T_{\tilde \nu}\|= 1$. Finally,
\begin{align*}
\|\tilde f - f\| &\leq \|\tilde f - f_1 \| + \|f_1 - f\|\\
& = \nor{ \sum_{j\in J} \tilde \beta_j \frac{\chi_{\tilde B_j}}{m_1(\tilde B_j)} - \sum_{j\in J} \tilde \beta_j \frac{\chi_{B_j}}{m_1(B_j)}}  +\eps \\
& \leq \sum_{j\in J} \tilde \beta_j \left(\nor{ \frac{\chi_{\tilde B_j}}{m_1(\tilde B_j)} -  \frac{\chi_{ B_j}}{m_1(\tilde B_j)} } +\nor{ \frac{\chi_{ B_j}}{m_1(\tilde B_j)} - \frac{\chi_{B_j}}{m_1(B_j)}}\right) +\eps \\
&=2\sum_{j\in J}\tilde \beta_j \frac{ m_1(B_j\setminus \tilde B_j)}{m_1(\tilde B_j)} +\eps\\
&\leq 2\sum_{j\in J}\tilde \beta_j \frac{ \frac{\eps}2 m_1(B_j)}{m_1(\tilde B_j)} +\eps\leq  \frac{\eps}{1-\eps/2}+\eps < 3\eps.\qedhere
\end{align*}
\end{proof}

We are now ready to present the proof of the main result in the case of finite regular positive Borel measures.

\begin{proof}[Proof of Proposition~\ref{maintheorem}]
Let $\delta_1 = \frac{\delta_2^3}{5\cdot 2^4}$, $\delta_2 = \frac{\delta_3^{12}}{3^2\cdot 2^{14}}$ and $\delta_3 = \left(\frac{\eps}{10}\right)^2$ for some $0<\eps<1$.
Suppose that $T\in \mathcal{L}(L_1(m))$ with $\|T\|=1$ and that there is an $f_0\in S_{L_1(m)}$ and $g_0\in S_{L_\infty(m)}$ such that $\inner{f_0, g_0}=1$ and  $|\inner{Tf_0, g_0}|>1-\frac{\delta_1^3}{2^6}$. Then there is an isometric isomorphism $\Psi$ from $L_1(m)$ onto itself such that $\Psi(f_0)=|f_0|$ and there is a scalar number $\alpha$ in $S_\mathbb{R}$ such that $|\inner{Tf_0, g_0}| = \inner{\alpha T f_0, g_0}$. Then letting $f_1= \Psi f_0$, $g_1 = (\Psi^{-1})^* g_0$ and $T_1= \alpha \Psi T\Psi^{-1}$,  we have
\[
\inner{ Sf_1, g_1} = \inner{\alpha \Psi T_0\Psi^{-1} \Psi f_0, (\Psi^{-1})^* g_0} = \inner{ \alpha Tf_0, g_0} > 1-  \frac{\delta_1^3}{2^6} \ \ \ \text{and} \ \ \ \inner{f_1, g_1} = \inner{\Psi f_0, (\Psi^{-1})^*g_0}=1.
\]
Since $\|T_1f_1\|>1-\delta \frac{\delta_1^3}{2^6}$, by Lemma~\ref{lem:real2}, there exists a norm-one bounded operator $T_\nu$ and a nonnegative simple function $f_2\in S_{L_1(m)}$ such that $\|T_1-T_\nu\|\leq \delta_1$, $\|f_2 - f_1 \|\leq 3\delta_1$ and $\frac{d|\nu|^1}{dm_1}(x)=1$ for all $x\in {\rm supp}(f_2)$.
Then
\begin{align*}
\inner{T_\nu f_2, g_1}& = \inner{T_1 f_1, g_1} - \inner{T_1 f_1 - T_1 f_2, g_1 } -\inner{T_1f_2- T_\nu f_2, g_1}\\
&\geq \inner{T_1f_1, g_1} - \|f_1 - f_2\| - \|T_1-T_\nu\|\\
&\geq 1-\frac{\delta_1^3}{2^6} - 3\delta_1 - \delta_1 \geq 1-5\delta_1=1-\frac{\delta_2^3}{16}.
\end{align*}
Notice also that
\[ \inner{f_2, g_1} = \inner{f_1, g_1} - \inner{f_1-f_2, g_1}\geq 1-\|f_1-f_2\| \geq 1-3\delta_1\geq 1-5\delta_1=1-\frac{\delta_2^3}{16}.\]
By Lemma~\ref{lem:real1} there are a nonnegative simple function $f_3\in S_{L_1(m)}$ and a function $g_3\in S_{L_\infty(m)}$ such that
\[ g_3(x) = \chi_{ {\rm supp} f_3}(x) + g_2(x)\chi_{\Omega \setminus {\rm supp } f_3}(x)\]
\[ \|f_2- f_3\|\leq \delta_2, \|g_3-g_1\|\leq \sqrt\delta_2 \ \ \ \text{and} \ \ \ \inner{f_3, g_3} =1.\]
So we have
\begin{align*}
\inner{T_\nu f_3, g_3}& = \inner{ T_\nu f_2, g_1} -\inner{ T_\nu f_2- T_\nu f_3, g_1} - \inner{T_\nu f_3, g_1-g_3}\\
& \geq 1-\frac{\delta_2^3}{16}-2\sqrt\delta_2 \geq 1-3\sqrt\delta_2 = 1-\frac{\delta_3^6}{2^7}.
\end{align*}
By Lemma~\ref{lem:real3}, there exist  $f_4\in S_{L_1(m)}$ and $g_4\in S_{L_\infty(m)}$ and an operator $T_4$ such that
\[ \inner{g_4, T_4f_4}=1=\|T_4\|,\   \|T_4-T_\nu\|\leq 2\delta_3,\  \|f_4-f_3\|\leq 3\delta_3, \|g_4- g_3\|\leq \sqrt\delta_3 \ \text{and} \\ \inner{f_4, g_4}=1.\]
So we have
\begin{align*}
\|T_4 - T_1 \|& \leq \|T_4 - T_\nu \| + \|T_\nu - T_1 \| \leq \delta_1 + 2\delta_3 \leq 3\delta_3\\
\|f_1 - f_4 \| &\leq \|f_1 - f_2\| + \|f_2 -f_3\|+\|f_3-f_4\|\leq 3\delta_1 + \delta_2 + 3\delta_3 \leq 10\delta_3\\
 \|g_1 - g_4 \| &\leq \|g_1 - g_3\|+ \|g_3-g_4\|\leq \delta_2 + \sqrt{\delta_3} \leq 2\sqrt{\delta_3}.
\end{align*}
Let $S = \alpha \Psi^{-1} T_4 \Psi$, $\tilde f = \Psi^{-1} f_4$ and $\tilde g =\Psi^* g_4$, then we have
\[ \norm{ T- S} =\norm{ T- \alpha \Psi^{-1} T_4 \Psi} = \norm{\alpha \Psi T \Psi^{-1} - T_4 } = \norm{ T_1 - T_4}\leq 3\delta_3\]
\[ \|f_0 - \tilde f \| = \|f_0 - \Psi^{-1}f_4\| = \|f_1 - f_4\|\leq 10\delta_3\]
\[ \|g_0 - \tilde f \| = \| g_0 - \Psi^* g_4\| = \|(\Psi^{-1})^*g_0 - g_4\| = \|g_1 -g_4\|\leq  2\sqrt{\delta_3}\]
\[ \inner{ \tilde f , \tilde g} = \inner{ \Psi^{-1} f_4, \Psi^* g_4} = \inner{f_4, g_4}=1\] and
\[ \left|\inner{S\tilde f, \tilde g}\right| =| \inner{ \alpha \Psi^{-1}T_4\Psi \Psi^{-1} f_4, \Psi^* g_4}|= |\alpha|=1.\]
This completes the proof.
\end{proof}

Finally, we may give the proof of the main result in full generality.

\begin{proof}[Proof of Theorem~\ref{thr:L_1-arbitrary}]
Notice that the Kakutani representation theorem  (see \cite{Lac} for a reference) says that for every $\sigma$-finite measure $\nu$, the space $L_1(\nu)$ is isometrically isomorphic to $L_1(m)$ for some positive Borel regular measure on a compact Hausdorff space. Then, by Proposition~\ref{maintheorem}, there is a universal function $\eps\longmapsto \eta(\eps)>0$ which gives the BPBp-$\nuu$ for $L_1(\nu)$ for every $\sigma$-finite measure $\nu$.

Fix $\eps>0$. Suppose that $T_0\in \mathcal{L}(L_1(\mu))$ with $v(T_0)=1$ and $(f_0, f_0^*)\in \Pi(L_1(\mu))$ satisfy
\[
|\inner{f_0^*, T_0f_0 } | > 1-\eta(\eps).
\]
Choose a sequence  $\{f_n\}$ in $L_1(\mu)$ such that $\sup_n \|T_0f_n\|=1$ and let $G$ be the closed linear span of
$$
\bigl\{ T^n f_m : n,m\in \mathbb{N}\cup \{0\} \bigr\}.
$$
As $G$ is separable, there is a dense subset $\{g_n\,:\,n\in \N\}$ of $G$ and let $E = \bigcup_{n=1}^\infty {\rm supp} ~g_n$, where ${\rm supp}~g_n$ is the support of $g_n$. Then the measure $\mu|_E$ is $\sigma$-finite. Let  $Y = \{ f \in L_1(\mu) : {\rm supp(f) }\subset E\}$ be a closed subspace of $L_1(\mu)$. It is clear that $L_1(\mu) = Y\oplus_1 Z$ and $Y$ is isometrically isomorphic to $L_1(\mu|_E)$. So $Y$ has the BPBp-$\nuu$ with $\eta(\eps)$.

Now, write $S_0 = T_0 |_Y: Y\longrightarrow Y$, consider $y_0 = f_0 \in S_Y$, $y_0^* = f_0^*|_Y\in S_{Y^*}$ and observe that
$y_0^*(y_0)=1$ and $|y_0^*(S_0y_0)| = |f_0^*(T_0f_0)|>1-\eta(\eps)$. Hence, there exist $S\in \mathcal{L}(Y)$ and $(\tilde y_0, \tilde y_0^*)\in \Pi(Y)$ such that
\[ |\tilde y_0^*(S \tilde y_0)|=1=v(S), \ \ \ \|S-S_0\|<\eps, \ \ \|y_0 - \tilde y_0\|<\eps \ \ \text{and} \ \ \|  y_0^*- \tilde y_0^*\|<\eps.\]
Finally consider the operator $T\in \mathcal{L}(L_1(\mu))$ given by
\[
T(y, z) = (Sy,0)+T_0(0, z)\qquad \bigl( (y, z) \in L_1(\mu) \equiv Y\oplus_1 Z\bigr)
.\]
We have $\|T\|=1$ (and so $v(T)=1$). Indeed,
$\|T(y,z)\| =\|(Sy, 0)\| + \|T_0(0,z)\|\leq \|y\|+\|z\|= \|(y, z)\| $ for all $(y,z)\in L_1(\mu)$ and $\|T(\tilde y_0, 0) \| = \|(S\tilde y_0,0)\|=\|S\tilde y_0\|=1$. Let $x=(\tilde y_0, 0)$ and $x^* = (\tilde y_0^*, f_0|_Z)$. Then $(x,x^*)\in \Pi(L_1(\mu))$. Moreover, we have
\begin{align*}
|x^*Tx|& = |\tilde y_0^* S y_0 | = 1 = v(T),\\
\|x-f_0\| &= \|y-y_0\|<\eps,\\
\|x_0^* - f_0^*\|& = \max\{ \|y-f_0^*|_Y\|, \|f_0^*|_Z - f_0^*|_Z\|\}= \|y^* - y_0^*\|<\eps\\ \intertext{and} \|T-T_0\|&= \sup_{\|y\|+\|z\|\leq 1} \|T(y,z) - T_0(y,z)\| = \sup_{\|y\|\leq 1} \|Sy- S_0y\| = \|S-S_0\|<\eps.
\end{align*}
This completes the proof.
\end{proof}

\section{Examples of spaces failing the Bishop-Phelps-Bollob\'{a}s property for numerical radius} \label{sec:counterexamples}

Our goal here is to prove that the density of numerical radius attaining operators does not imply the BPBp-$\nuu$. Actually, we will show that among separable spaces, there is no isomorphic property implying the BPBp-$\nuu$ other than finite-dimensionality.

We need to relate the BPBp-$\nuu$ with the Bishop-Phelps-Bollob\'{a}s property for operators which, as mentioned in the introduction, was introduced in \cite{AAGM2}. A pair $(X, Y)$ of Banach spaces has the {\it Bishop-Phelps-Bollob\'as property for operators} (in short, \emph{BPBp}), if given $\varepsilon >0$ there exists $\eta(\varepsilon)>0$ such that given $T\in {\mathcal{L}(X,Y)}$ with $\|T\|=1$ and $x \in S_X$ such that $\|Tx\|>1-\eta (\varepsilon)$, then there exist $z\in S_X$ and $S\in {\mathcal{L}(X,Y)}$ satisfying $$
\|S\|=\|Sz\|=1,\quad \|x - z\|<\eps\quad \text{and} \quad \|T-S\|<\eps.
$$
We refer the reader to \cite{AAGM2,ACKLM,CKLM} and references therein for more information and background. Among the interesting results on the BPBp, we emphasize that a pair $(X,Y)$ when $X$ is finite-dimensional does not necessarily have the BPBp. For instance, if $Y$ is a strictly convex space which is not uniformly convex, then the pair $(\ell_1^{(2)},Y)$ fails to have the BPBp (this is contained in \cite{AAGM2}, see \cite[Section 3]{ACKLM}).

The next result relates the BPBp-$\nuu$ with the BPBp for operators in a particular case. We will deduce our example from it.

\begin{theorem}\label{L1}
If $L_1(\mu)\oplus_1 X$ has the BPBp-$\nuu$, then the pair $(L_1(\mu),X)$ has the BPBp for operators.
\end{theorem}

Before proving this proposition, we will use it to get the main examples of this section. The first example shows that the density of numerical radius attaining operators does not imply the BPBp-$\nuu$.

\begin{example}\label{exam-first}
{\slshape There is a reflexive space (and so numerical radius attaining operators on it are dense) which fails to have the BPBp-$\nuu$.\ } Indeed, let $Y$ be a reflexive separable space which is not superreflexive and we may suppose that $Y$ is strictly convex. Observe that $Y$ cannot be uniformly convex since it is not superreflexive. Now, $X=\ell_1^{(2)}\oplus_1 Y$ is reflexive, but the pair $(\ell_1^{(2)},Y)$ fails the BPBp since $Y$ is strictly convex but not uniformly convex \cite[Corollary~3.3]{ACKLM}. Therefore, Theorem~\ref{L1} gives us that $X$ does not have the BPBp-$\nuu$.
\end{example}

The example above can be extended to get the result that every infinite-dimensional separable Banach space can be renormed to fail the BPBp-$\nuu$. This follows from the fact that every infinite-dimensional separable Banach space can be renormed to be strictly convex but not uniformly convex (this result can be proved ``by hand''; an alternative categorical argument for it can be found in \cite{God} and references therein). With a little more of effort, we may get the main result of the section.

\begin{theorem}\label{thr-renorming}
Every infinite-dimensional separable Banach space can be renormed to fail the weak-BPBp-$\nuu$ (and so, in particular, to fail the BPBp-$\nuu$).
\end{theorem}

We need the following result which is surely well known. As we have not found a reference, we include a nice and easy proof kindly given to us by Vladimir Kadets. We recall that, given a Banach space $Y$, the set of all equivalent norms on $Y$ can be viewed as a metric space using the Banach-Mazur distance.

\begin{lemma}\label{lemma-vladimir}
Let $Y$ be an infinite-dimensional separable Banach space. Then the set of equivalent norms on $Y$ which are strictly convex and are not (locally) uniformly convex is dense in the set of all equivalent norms on $Y$ (with respect to the Banach-Mazur distance).
\end{lemma}

\begin{proof}
Fix $e\in S_Y$ and $e_1^*\in S_{Y^*}$ such that $e_1^*(e)=1$. For a fixed $\eps\in (0,1/2)$, denote
$$
q(y)=\max\bigl\{(1-\eps)\|y\|,|e_1^*(y)|\bigr\} \quad (y\in Y).
$$
Evidently,
$(1-\eps)\|y\|\leq q(y)\leq \|y\|$ for every $y\in Y$. Fix a sequence $\{e_k^*\,:\,k\geq 2\}$ of norm-one functionals separating the points of $Y$, and denote
$$
p(y)=\sqrt{\sum_{k=1}^\infty \frac{1}{2^k}|e_k^*(y)|^2} \quad (y\in Y).
$$
Then, $p$ is a strictly convex norm on $Y$, $p(e)\geq \frac{1}{\sqrt{2}}$ and $p(y)\leq \|y\|$ for all $y\in X$. Finally, write
$$
\|y\|_1=(1-\eps)q(y) + \eps \dfrac{p(y)}{p(e)} \quad (y\in Y).
$$
Then, $\|\cdot\|_1$ is a strictly convex norm on $Y$ and
$$
(1-\eps)^2\|y\| \leq \|y\|_1 \leq (1+\eps)\|y\| \quad (y\in Y).
$$
We will finish the proof by showing that $\|\cdot\|_1$ is not uniformly convex (actually, it is not locally uniformly convex). Indeed, for each $n\in \N$ we select $y_n\in \bigcap_{k=1}^n \ker e_k^*$ with $\|y_n\|=1$ and consider $e_n=e+\frac{\eps}{4}y_n$. Then, $q(e)=1$, $q(e_n)=1$, and $q(e+e_n)=2$. At the same time, $p(y_n)\longrightarrow 0$, so $p(e_n)\longrightarrow p(e)$ and $p(e+e_n)\longrightarrow 2 p(e)$. Consequently,
$$
\|e\|_1=1,\quad \|e_n\|_1\longrightarrow 1,\quad \text{and}\quad \|e+e_n\|_1\longrightarrow 2,
$$
but $\|e-e_n\|_1=\frac{\eps}{4}\|y_n\|_1\geq (1-\eps)^2\frac{\eps}{4}$, which means the absence of local uniform convexity at $e$.
\end{proof}

\begin{proof}[Proof of Theorem~\ref{thr-renorming}]
Let $X$ be an infinite-dimensional separable Banach space. Take a closed subspace $Y$ of $X$ of codimension two.
By \cite[Propositon~2]{FMP03}, the map carrying every equivalent norm on $Y$ to its numerical index is continuous and so, the set of values of the numerical index of $Y$ up to reforming is a non-trivial interval \cite[Theorem~9]{FMP03}. Then Lemma~\ref{lemma-vladimir}  allows us to find an equivalent norm $|\cdot|$ on $Y$ in such a way that $(Y,|\cdot|)$ is strictly convex, is not uniformly convex, and $n(Y,|\cdot|)>0$. Now, the space $\widetilde{X}=\ell_1^{(2)}\oplus_1 (Y,|\cdot|)$ is an equivalent renorming of $X$ which does not have the BPBp-$\nuu$ (indeed, otherwise, the pair $\bigl(\ell_1^{(2)},(Y,|\cdot|)\bigr)$ would have the BPBp for the operator norm and so, $(Y,|\cdot|)$ would be uniformly convex by \cite[Corollary~3.3]{ACKLM}, a contradiction.) Moreover, as
$$
n\bigl(\widetilde{X}\bigr)=\min\bigl\{n(\ell_1^{(2)}),n(Y,|\cdot|)\bigr\}>0
$$
(see \cite[Proposition~2]{KMP} for instance), $\widetilde{X}$ also fails the weak-BPBp-$\nuu$ by Proposition~\ref{prop-nxpositive}.
\end{proof}

To finish the section with the promised proof of Theorem~\ref{L1}, we first state the following stability result.

\begin{lemma}\label{thm:stability} Let $X=\big[\bigoplus_{k=1}^\infty X_k\big]_{c_0}$ or $\big[\bigoplus_{k=1}^\infty X_k\big]_{\ell_1}$. If $X$ has the Bishop-Phelps-Bollob\'as property for numerical radius with a function $\eta$, then each Banach space $X_i$ has the Bishop-Phelps-Bollob\'as property for numerical radius with $\eta_{nu}(X_i)\geq \eta$. That is, $\inf_i \eta_{nu}(X_i)(\eps)\geq \eta_{nu}(X)(\eps)$ for all $0<\eps<1$.
\end{lemma}
\begin{proof}
Let $P_i~:~X\longrightarrow X_i$ and $P_i'~:~X^*\longrightarrow X^*_i$ be the natural projections, and let $Q_i~:~X_i\longrightarrow X$ and $Q_i'~:~X^*_i\longrightarrow X^*$ be the natural embeddings.

Assume that an operator $T_i~:~X_i\longrightarrow X_i$ and a pair $(x_i,x^*_i)\in \Pi(X_i)$ satisfy that
$$v(T_i)=1\ \ \ \text{~and~}\ \ \ |x^*_i T_i x_i|>1-\eta(\eps).$$
We define an operator $T~:~X\longrightarrow X$ and $(x,x^*)\in \Pi(X)$  by
$$T=Q_i\circ T_i\circ P_i\ \ \ \text{~and~}\ \ \  (x,x^*)=(Q_ix_i,Q_i'x^*_i),$$
then clearly we have that
$$|x^* T x|=|x^*_i T_i x_i|>1-\eta(\eps).$$
From the assumption, there exist $S~:~X\longrightarrow X$ and a pair $(y,y^*)\in \Pi(X)$ such that
$$|y^*Sy|=1=v(S),~\|S-T\|<\varepsilon,~\|y^*-x^*\|<\varepsilon, \text{~and~}\|y-x\|<\varepsilon.$$
Since this clearly shows that
$$\|P_i\circ S\circ Q_i-T_i\|<\varepsilon,~\|P_i'y^*-x^*_i\|<\varepsilon,\text{~and~}\|P_iy-x_i\|<\varepsilon,$$
we only need to show that $|P_i'y^* (P_i\circ S \circ Q_i) P_iy|=1$.\\

We first show the case of $c_0$ sum. Since $\|P_jy\|=\|P_jy-P_jx\|\leq\|y-x\|<\varepsilon$ for every $j\neq i$, we have

 \begin{align*}
1
&=y^*(y)=\sum_{j\in\mathbb{N}} P_j'y^*(P_jy)\leq \sum_{j\in\mathbb{N}} \|P_j'y^*\|\|P_jy\|\\
&\leq \|P_i'y^*\|+ \varepsilon \sum_{j\in\mathbb{N},~j\neq i} \|P_j'y^*\|\leq\|y^*\|= 1.
\end{align*}
This shows that $\|P_i'y^*\|=1$ and $P_j'y^*=0$ for every $j\neq i$.  So $y^* = Q_i'P_i'y^*$ and $P_i'y^*(P_i y)=1$. This and the fact that $\|y-Q_iP_iy\|<\eps$ imply that $(Q_iP_i y+ {{1}\over{\varepsilon}}(y-Q_iP_iy),Q_i'P_i'y^*)\in \Pi(X)$. So we get that $(Q_i'P_i'y^*)S(Q_iP_i y+ {{1}\over{\varepsilon}}(y-Q_iP_iy))\leq v(S)=1$.
Hence, we have
\begin{align*}
1
&=|y^*Sy|=|(Q_i'P_i'y^*)Sy|\\
&=\left|(1-\varepsilon) (Q_i'P_i'y^*)S (Q_iP_i y)+\varepsilon(Q_i'P_i'y^*)S\left(Q_iP_i y+ {{1}\over{\varepsilon}}(y-Q_iP_iy)\right)\right|\leq 1,
\end{align*}
and so we get $|P_i'y^* (P_i\circ S \circ Q_i) P_iy|=|(Q_i'P_i'y^*)S (Q_iP_i y)|=1$.\\

We next show the case of $\ell_1$ sum. The proof is almost the same as that of the $c_0$ case. However,  for the sake of completeness,  we provide it here.

Since $\|P_j'y^*\|=\|P_j'y^*-P_j'x^*\|\leq\|y^*-x^*\|<\varepsilon$ for every $j\neq i$, we have

 \begin{align*}
1
&=y^*(y)=\sum_{j\in\mathbb{N}} P_j'y^*(P_jy)\leq \sum_{j\in\mathbb{N}} \|P_j'y^*\|\|P_jy\|\\
&\leq \|P_iy\|+ \varepsilon \sum_{j\in\mathbb{N},~j\neq i} \|P_jy\|\leq \|y\|=1,
\end{align*}
which shows $\|P_iy\|=1$ and $P_jy=0$ for every $j\neq i$. Since this implies $(Q_iP_iy,Q_i'P_i' y^*+ {{1}\over{\varepsilon}}(y^*-Q_i'P_i'y^*))\in \Pi(X)$, we get that $\left|(Q_i'P_i' y^*+ {{1}\over{\varepsilon}}(y^*-Q_i'P_i'y^*))S(Q_iP_iy)\right|\leq v(S)=1$.
Hence, we have
\begin{align*}
1
&=|y^*Sy|=|y^*S(Q_iP_iy)|\\
&=|(1-\varepsilon) (Q_i'P_i'y^*)S (Q_iP_i y)+\varepsilon(Q_i'P_i' y^*+ {{1}\over{\varepsilon}}(y^*-Q_i'P_i'y^*))S(Q_iP_iy)\leq 1,
\end{align*}
and so we get $|P_i'y^* (P_i\circ S \circ Q_i) P_iy|=|(Q_i'P_i'y^*)S (Q_iP_i y)|=1$.
\end{proof}

\begin{proof}[Proof of Theorem~\ref{L1}] Note that $\eta_{\nuu}(L_1(\mu)\oplus_1 X)(\varepsilon)\longrightarrow 0$  as $\varepsilon\longrightarrow 0$. Fix $0<\eps_0<1$ and choose $0<\eps<1$ such that $6\eps + \eta_{\nuu}(L_1(\mu)\oplus_1 X)(\varepsilon)<\eps_0$. Let $\eta(\varepsilon_0)=\eta_{\nuu}(L_1(\mu)\oplus_1 X)(\varepsilon)$.

Suppose that $T_0\in \mathcal{L}(L_1(\mu),X)$ with $\|T_0\|=1$ and $f_0\in S_{L_1(\mu)}$ satisfy
$$
\|T_0f_0\|>1-\eta(\varepsilon_0).
$$
For any measurable subset $B$, let $L_1(\mu|_B)=\{ f|_B : f\in L_1(\mu)\}$ with the norm $\|f|_B\|=\|f\chi_B\|_1$. Then it is easy to see that $L_1(\mu|_B)$ is isometrically isomorphic to a complemented subspace of $L_1(\mu)$. Let  $P_B: L_1(\mu) \longrightarrow L_1(\mu|_B)$ be the restriction defined by $P_B(f) = f|_B$ for all $f\in L_1(\mu)$ and let $J_B:L_1(\mu|_B)\longrightarrow L_1(\mu)$ be the extension defined by $J_B(f)(\omega) = f(\omega)$ if $\omega\in B$ and $J_B(f)(\omega) =0$ otherwise. It is clear that $P_BJ_B = {\rm Id}_{L_1(\mu|_B)}$ and $J_BP_B(f) = f\chi_B$  for all $f\in L_1(\mu)$. Notice also that $L_1(\mu)$ is isometrically isomorphic to $L_1(\mu|_B)\oplus_1 L_1(\mu|_{B^c})$.

Let $A={\rm supp} f_0$ and $g_0=P_Af_0$.  Then $\|T_0J_Ag_0\|= \|T_0f_0\|>1-\eta(\varepsilon_0)>0$ and define the operator $T_A~:~L_1(\mu|_A)\longrightarrow X$ by $T_Af={{T_0J_Af}\over{\|T_0J_A\|}}$ for every $f\in L_1(\mu|_A)$. Then,
 $$
\|T_Ag_0\|\geq \|T_0f_0\|>1-\eta(\varepsilon_0).
$$
Since $\mu|_A$ is $\sigma$-finite, $L_1(\mu|_A)^* = L_\infty(\mu|_A)$. Let $g_0^*\in S_{L_\infty(\mu|_A)}$ be a function such that $\inner{g_0^*,g_0}=1$, and choose $x_0^*\in S_{X^*}$ such that $x_0^*(T_Ag_0)=\|T_Ag_0\|$. Define the operator $S_0\in \mathcal{L}(L_1(\mu|_A)\oplus_1 X)$ by
$$
S_0(f,x)=(0,T_Af) \qquad \bigl((f,x)\in L_1(\mu|_A)\oplus_1 X\bigr)
$$
and observe that $\|S_0\|=v(S_0)=1$.  Indeed,
 \begin{align*}\|S_0\|\leq 1=\|T_A\|&=\sup \{|x^*T_Af|~:~x^*\in S_{X^*}, f\in S_{L_1(\mu|_A)}\}\\
 &=\sup \{|(f^*,x^*)S_0(f,x)| ~:~((f^*,x^*),(f,x))\in \Pi(L_1(\mu|_A)\oplus_1 X)\}\\
 &=v(S_0)\leq \|S_0\|.
 \end{align*}
It is immediate that
$$
(g_0^*,x_0^*)S_0(g_0,0)=x^*_0(T_Ag_0)=\|T_Ag_0\|>1-\eta(\varepsilon_0).
$$
By Lemma~\ref{thm:stability}, $L_1(\mu|_A)\oplus_1 X$ has the BPBp-$\nuu$ with the function $\eta$. Therefore, there exist $S_1\in {\mathcal{L}(L_1(\mu|_A)\oplus_1 X)}$, $(g_1,x_1)\in S_{L_1(\mu|_A)\oplus_1 X}$ and $(g^*_1,x^*_1)\in S_{L_\infty(\mu|_A)\oplus_\infty X^*}$ such that
$$
\|(g_1,x_1)-(g_0,0)\|<\eps,\quad \|(g^*_1,x^*_1)-(f_0^*,x_0^*)\|<\eps, \quad \|S_1-S_0\|<\varepsilon,
$$
$$\inner{(g^*_1,x^*_1),(g_1,x_1)}=1 \quad \text{~and~} \quad |(g^*_1,x^*_1)S_1(g_1,x_1)|=v(S_1)=1.
$$

\emph{Claim $1$}. We claim that $x_1=0$. \\
Otherwise,
   \begin{align*}
   1=\re\inner{(g^*_1,x^*_1),(g_1,x_1)}& =\|g_1\|~\re\inner{(g^*_1,x^*_1), \frac{(g_1,0)}{\|g_1\|}} + \|x_1\|~\re\inner{(g^*_1,x^*_1), \frac{(0,x_1)}{\|x_1\|} }\leq 1.
   \end{align*}
We deduce that
$$
\left(\frac{(g_1,0)}{\|g_1\|},(g^*_1,x^*_1)\right), \ \left(\frac{(0,x_1)}{\|x_1\|},(g^*_1,x^*_1)\right)\in \Pi(L_1(\mu|_A)\oplus_1 X).
$$
Since $\left\|S_1(\tfrac{(0,x_1)}{\|x_1\|})\right\| = \left\|(S_1-S_0)\left(\tfrac{(0,x_1)}{\|x_1\|}\right)\right\|<\eps$, we get that
 \begin{align*}1 &= | \inner{(g^*_1,x^*_1), S_1(g_1,x_1)}| \\ &= \left|\|g_1\|\inner{(g^*_1,x^*_1), S_1\left(\frac{(g_1,0)}{\|g_1\|}\right)} + \|x_1\|\inner{(g^*_1,x^*_1), S_1\left(\frac{(0,x_1)}{\|x_1\|}\right)}\right|\\
& \leq \|g_1\|v(S_1)+\eps\|x_1\|< \|g_1\|+ \|x_1\|=1,
\end{align*} a contradiction. This proves the claim.

We define the operator $S_2~:~L_1(\mu|_A)\oplus_1X\longrightarrow L_1(\mu|_A)\oplus_1X$ by $S_2(f,x)=S_1(f,0)$ for every $f\in L_1(\mu|_A)$ and for every $x\in X$. Then we have
$$
v(S_2) =|(g^*_1,x^*_1)S_2(g_1,0)|=1 \ \ \text{and}\ \ \| S_1 - S_2\|\leq \eps.
$$
Indeed, from Claim~1, we have
$$
v(S_1)=|(g^*_1,x^*_1)S_1(g_1,x_1)|=|(g^*_1,x^*_1)S_1(g_1,0)|=
|(g^*_1,x^*_1)S_2(g_1,0)|\leq v(S_2).
$$
On the other hand, we have that
$$
|(f^*,x^*)S_2(f,x)|=|(f^*,x^*)S_1(f,0)|\leq \|f\| v(S_1) \leq v(S_1)
$$
for every $((f^*,x^*),(f,x))\in \Pi(L_1(\mu|_A)\oplus_1 X)$. So $v(S_2)\leq v(S_1)$. Also, $$
\|S_1-S_2\|\leq \sup_{x\in S_X}\|S_1(0,x)\|=\sup_{x\in S_X}\|S_1(0,x)-S_0(0,x)\|\leq \varepsilon.
$$
 \emph{Claim $2$}. There exists an operator $S_3~:~L_1(\mu|_A)\oplus_1 X\longrightarrow L_1(\mu|_A)\oplus_1 X$ such that $\|S_3(g_1,0)\|=\|S_3\|=1$, $S_3(0,x)=0$, $S_3(f,x)\in \{0\}\oplus_1 X$ for every $(f,x)\in L_1(\mu|_A)\oplus_1 X$ and $\|S_3-S_2\|<4\varepsilon$.

Indeed, using the trivial decomposition, write $S_1=(D_1,D_2)$, where $D_1:L_1(\mu|_A)\oplus_1 X\longrightarrow L_1(\mu|_A)$ and $D_2:L_1(\mu|_A)\oplus_1 X\longrightarrow X$. We have that
\begin{align*}
\sup\bigl\{ &|g^*D_1 (g_1,0)+x^*D_2(g_1,0)| \,:\, x^*\in S_{X^*}, \inner{g^*, g_1}=1, g^*\in S_{L_\infty(\mu|_A)}\bigr\} \\
&= \sup \bigl\{ |g^*D_1 (g_1,0)| +\|D_2(g_1,0)\|  \,:\, \inner{g^*, g_1}=1, g^*\in S_{L_\infty(\mu|_A)}\bigr\} \\
&= \sup \bigl\{ |g^*D_1 (g_1,0)|   \,:\, \inner{g^*, g_1}=1, g^*\in S_{L_\infty(\mu|_A)}\bigr\} \ +\  \|D_2(g_1,0)\|\\
&\leq v(S_2)=\bigl|(g^*_1,x^*_1)S_2(g_1,0)\bigr|= \bigl|g^*_1D_1(g_1,0)+x^*_1D_2(g_1,0)\bigr|.
\end{align*}
This implies that $$|x^*_1D_2(g_1,0)|=\|D_2(g_1,0)\|$$
and
\begin{align*}
 |g^*_1D_1(g_1,0)| = \sup \{ |g^*D_1 (g_1,0)|   : \inner{g^*, g_1}=1, g^*\in L_\infty(\mu|_A)\}.
\end{align*}
Therefore, $|g_1^*|$ equals $1$ on the support of $D_1(g_1,0)$. As $|\inner{g_1^*,g_1}|=1$, we also have that $|g_1^*|$ equals $1$ on the support of $g_1$. Changing the values of $g_1^*$ by the ones of $f_0^*$ on $A\setminus \bigl({\rm supp}( D_1(g_1,0))\cup {\rm supp}(g_1)\bigr)$, we may and do suppose that $|g_1^*|=1$ on the whole $A$.

We also have $\|D_2(g_1, 0)\|>0$  Indeed,
\begin{align*}
\|S_2(g_1, 0) - S_0(g_0, 0)\| &\leq \|S_2(g_1, 0) - S_0(g_1, 0)\| + \|S_0(g_1, 0) - S_0(g_0, 0)\|\\
&< 2\eps+\eps = 3\eps.
\end{align*}
So we have
\begin{align*}
\|D_2(g_1,0)-T_Ag_0\|& \leq \|D_1(g_1,0)\|+\|D_2(g_1,0)-T_Ag_0\| \\&=
\|(D_1(g_1, 0), D_2(g_1,0) ) - (0, T_Ag_0)\|\\
&=\|S_2(g_1, 0) - S_0(g_0, 0)\| <3\eps
\end{align*}
and $\|D_2(g_1, 0)\|> \|T_Ag_0\|-3\eps\geq 1-\eta(\eps_0)-3\eps>0$.

Finally define the operator $S_3$ by
$$
S_3(f,x)=\left(0,D_2(f,0)+g^*_1(D_1(f,0)) \tfrac{D_2(g_1,0)}{x^*_1D_2(g_1,0)}\right)\ \  \  \ \text{~for~}\ \  (f,x)\in L_1(\mu|_A)\oplus_1 X.
$$
It is clear that $\|S_3\|\leq \sup_{f\in S_{L_1(\mu|_A)}} (\|D_2(f,0)\|+|g^*_1D_1(f,0)|)$. Notice also that
\begin{align*}
\|D_1(f,0)\| &\leq \|D_1(f,0)\|+\|D_2(f,0)-T_Af\|\\ &=
 \|(D_1(f, 0), D_2(f,0) ) - (0, T_Af)\|\\ &=\|S_2(f, x) - S_0(f, x)\|
 \end{align*}
for all $(f, x)\in L_1(\mu|_A)\oplus_1 X$. Hence we have
 $$
 \|S_3 - S_2\|=2\sup_{f\in S_{L_1(\mu|_A)}}\|D_1(f,0)\|   \leq 2\|S_2-S_0\|<4\eps.
 $$
On the other hand, let $G~:~ L_1(\mu|_A) \longrightarrow L_1(\mu|_A)$ be defined by $G(f)=\overline{g^*_1}f$ for every $f\in L_1(\mu|_A)$. Then, we have
\begin{align*}
v(S_2)
&=\sup\{|z^*S_2z|~:~(z,z^*)\in\Pi(L_1(\mu|_A)\oplus_1 X)\}\\
&\geq \sup\left\{ \left|x^*D_2\left(G\left(\frac1{\mu(C)}\chi_{C}\right),0\right)+g^*_1D_1\left(G\left(\frac1{\mu(C)}\chi_{C}\right),0\right)\right|~:~{x^*\in S_{X^*}},~C\in \Sigma_A, \mu(C)>0\right\}\\
&=\sup\left\{ \left\|D_2\left(G\left(\frac1{\mu(C)}\chi_{C}\right),0\right)\right\|+\left|g^*_1D_1\left(G\left(\frac1{\mu(C)}\chi_{C}\right),0\right)\right|~:~C\in  \Sigma_A, \mu(C)>0\right\},
\end{align*} where $\Sigma_A $ is the family of measurable subsets of $A$.

Hence, for any simple function $s=\sum_{i=1}^n \frac{\alpha_i}{\mu(A_i)}\chi_{A_i}\in S_{L_1(\mu|_A)}$, where $\{A_i\}_i$ is a family of disjoint measurable subsets with strictly positive measure, we have
\begin{align*}
v(S_2)
 &\geq\sum_{i=1}^n|\alpha_i| \left(\left\|D_2\left(G\left(\frac{1}{\mu(A_i)}\chi_{A_i}\right),0\right)\right\| +\left|g^*_1D_1\left(G\left(\frac{1}{\mu(A_i)}\chi_{A_i}\right),0\right)\right|\right)\\
 &\geq\|D_2(G(s),0)\|+|g^*_1D_1(G(s),0)|.
\end{align*}
Since $|g^*_1|=1$, $G$ is an isometric isomorphism, so for each $f\in S_{L_1(\mu|_A)}$ there exists a sequence of norm-one simple functions $(s_k)$ such that $G(s_k)$ converges to $f$. Therefore, $$
v(S_2) \geq \sup_{f\in S_{L_1(\mu|_A)}} \bigl(\|D_2(f,0)\|+|g^*_1D_1(f,0)|\bigr)\geq \|S_3\|.
$$
On the other hand, we have that
\begin{align*}
\|S_3\|
\geq |(g^*_1,x^*_1)S_3(g_1,0)|
=|x^*_1D_2(g_1,0)+g^*_1D_1(g_1,0)|=v(S_2)=1.
\end{align*}
Therefore, $1=\|S_3\|=\|S_3(g_1,0)\|$ which proves Claim ~2.\\

Finally, we write $S_3 = (0, \tilde{T})$ for a suitable $\tilde{T}: L_1(\mu|_A)\oplus_1X \longrightarrow X$  and we define the operator $T_1~:~L_1(\mu)\longrightarrow X$ by
$$
   T_1(f)=T_0(f\chi_{A^c}) + \tilde{T}(P_Af, 0) \ \ \ \text{for every} \ \ f\in L_1(\mu).
$$
Then, we have
$$
\|T_1(f)\|\leq \|T_0\|\|f\chi_{A^c}\| + \|\tilde{T}\|\|f\chi_A\|=\|f\|
$$
for every $f\in L_1(\mu)$, so $\|T_1\|\leq 1$. Also,
$$
\|T_1(J_Ag_1)\|=\|S_3(g_1,0)\|=\|S_3\|=1,
$$
so $T_1$ attains its norm on $J_Ag_1\in L_1(\mu)$, and
$$
\|J_Ag_1-f_0\|=\|g_1-g_0\|<\eps.
$$
We also have that for any $f\in S_{L_1(\mu)}$,
\begin{align*}
\|T_0(f)-T_1(f)\| & = \|T_0(f\chi_A ) - \tilde{T}(P_A f, 0) \|\\
&\leq \|T_0(J_AP_Af) - T_A(P_Af) \| + \|T_A(P_Af) -  \tilde{T}(P_A f, 0) \|\\
&\leq \|T_0J_A - T_A\| + \|S_0 - S_3\|\\
&<\eta(\eps_0) + 6\eps.
\end{align*} Hence $\|T_0 - T_1\| \leq \eta(\eps_0) + 6\eps <\eps_0$. This completes the proof.
\end{proof}

{\bf Conflict of Interests.}
The authors declare that there is no conflict of interests regarding the publication of this article.

\section*{Acknowledgment} We thank Gilles Godefroy and Rafael Pay\'{a} for fruitful conversations about the content of this paper, and Vladimir Kadets for providing Lemma~\ref{lemma-vladimir}. We also appreciate anonymous referees for careful reading and fruitful suggestions about revision.

\end{document}